\documentclass{article}

\usepackage[utf8]{inputenc} 
\usepackage[T1]{fontenc}    
\usepackage{hyperref}       
\usepackage{url}            
\usepackage{booktabs}       
\usepackage{amsfonts}       
\usepackage{dsfont}
\usepackage{amssymb}
\usepackage{amsmath}
\usepackage{graphicx}
\usepackage{color}
\usepackage{mathtools}
\usepackage{natbib}

\newcommand{\Norm}[1]{\left\|#1\right\|}

\newcommand{\BlackBox}{\rule{1.5ex}{1.5ex}}  
\newcommand{\qed}{\hfill\BlackBox\\[2mm]}
\newenvironment{proof}{\par\noindent{\bf Proof\ }}{\qed}

\newtheorem{lemma}{Lemma}[section]
\newtheorem{theorem}[lemma]{Theorem}

\newtheorem{proposition}[lemma]{Proposition}
\newtheorem{definition}[lemma]{Definition}

\DeclarePairedDelimiter\floor{\lfloor}{\rfloor}
\newcommand{\bE}{\mathbb{E}}
\newcommand{\bN}{\mathbb{N}}
\newcommand{\bP}{\mathbb{P}}
\newcommand{\R}{\mathbb{R}}
\newcommand{\bR}{\mathbb{R}}

\newcommand{\cC}{\mathcal{C}}
\newcommand{\cF}{\mathcal{F}}
\newcommand{\cG}{\mathcal{G}}
\newcommand{\cL}{\mathcal{L}}
\newcommand{\cO}{\mathcal{O}}
\newcommand{\cP}{\mathcal{P}}
\newcommand{\cT}{\mathcal{T}}
\newcommand{\cX}{\mathcal{X}}
\newcommand{\cY}{\mathcal{Y}}

\newcommand{\un}{\mathds{1}}

\newcommand{\bayes}{f^*}

\newcommand{\ERM}[1]{\widehat{f}_{#1}}
\newcommand{\ERMag}[1]{\widehat{f}^{\mathrm{ag}}_{#1}}
\newcommand{\ERMho}[1]{\widehat{f}^{\,\mathrm{ho}}_{#1}}

\newcommand{\ClassRule}{G}

\newcommand{\HO}[2]{\text{HO}_{#1}\left(#2\right)}
\newcommand{\ho}{\text{ho}}

\newcommand{\cRh}{\widehat{\mathcal{R}}} 

\DeclareMathOperator*{\argmin}{argmin}
\DeclareMathOperator*{\argmax}{argmax}
\DeclareMathOperator{\Var}{Var}

\title{%
Cross-validation improved by aggregation: Agghoo
}
\author{
Guillaume Maillard \\
Ecole Normale Sup\'erieure\\
Paris 75005, France  \\
\texttt{guillaume.maillard@ens.fr} \\
\and
Sylvain Arlot \\
Laboratoire de Math\'ematiques d'Orsay, \\
Univ. Paris-Sud, CNRS, Universit\'e Paris-Saclay, \\
91405 Orsay, France \\
\texttt{sylvain.arlot@math.u-psud.fr} \\
\and
Matthieu Lerasle \\
Laboratoire de Math\'ematiques d'Orsay, \\
Univ. Paris-Sud, CNRS, Universit\'e Paris-Saclay, \\
91405 Orsay, France \\
\texttt{matthieu.lerasle@math.u-psud.fr} \\
}

\begin{document}
\maketitle

\begin{abstract}
Cross-validation is widely used for selecting among a family of learning rules. 
This paper studies a related method, called aggregated hold-out (Agghoo), 
which mixes cross-validation with aggregation; 
Agghoo can also be related to bagging. 
According to numerical experiments, Agghoo can improve significantly 
cross-validation's prediction error, at the same computational cost; 
this makes it very promising as a general-purpose tool for prediction. 
We provide the first theoretical guarantees on Agghoo, 
in the supervised classification setting, 
ensuring that one can use it safely: 
at worse, Agghoo performs like the hold-out, up to a constant factor. 
We also prove a non-asymptotic oracle inequality, 
in binary classification under the margin condition, 
which is sharp enough to get (fast) minimax rates. 
%
%
%
\end{abstract}

\section{Introduction}

%
Machine learning rules almost always depend on some hyperparameters, 
whose choice has a strong impact on the final performance. 
For instance, nearest-neighbor rules  \citep{Bia_Dev:2015} 
depend on the number $k$ of neighbors 
and on some distance measure over the feature space. 
Kernel methods \citep{Sch_Smo:2001} require to choose an appropriate kernel. 
A third example, among many others, is given by regularized empirical risk minimization rules, 
such as support vector machines \citep{Ste_Chr:2008} or the Lasso \citep{Tib:1996,Buh_vdG:2011}, 
which all depend on some regularization parameter. 
More generally, the problem of choosing from data among a family of learning rules 
is central to machine learning, 
including model selection \citep{Bur_And:2002,Mas:2003:St-Flour} 
and when one hesitates between different kinds of rules ---for instance, 
support vector machines or random forests.

%
In supervised learning, 
cross-validation (CV) is a general, efficient and classical answer to this problem 
\citep{Arl_Cel:2010:surveyCV}. 
It relies on the idea of splitting data into a training sample 
---used for training a predictor with each rule in competition--- 
and a validation sample 
---used for assessing the performance of each predictor. 
This leads to an estimator of the risk 
---the hold-out estimator when data are split once, 
the CV estimator when an average is taken over several data splits---, 
which can be minimized for selecting among a family of competing rules.

%
A completely different strategy, called aggregation, 
is to \emph{combine} the predictors obtained with all candidate rules 
\citep{Nem:2000,Yan:2001,Tsy:2004}.  
%
%
A major interest of aggregation is that it builds a learning rule 
in a much larger set that the family of rules in competition; 
therefore, it can sometimes yield a better performance than 
the best of all rules. 
%
%
Aggregation is the keystone of ensemble methods \citep{Die:2000}, 
among which we can mention bagging \citep{Bre:1996a}, 
AdaBoost \citep{Fre_Sch:1997} 
and random forests \citep{Bre:2001,Bia_Sco:2016:TEST}.

\medbreak

%
This paper studies a mix of cross-validation and aggregation ideas, 
called \emph{aggregated hold-out} (Agghoo). 
Data are split several times; 
for each split, the hold-out selects one predictor; 
then, the predictors obtained with the different splits are aggregated. 
This procedure is as general as cross-validation 
and it has the same computational cost. 
Moreover, it seems to be folklore knowledge among practicioners that 
Agghoo (or its variants) 
performs better than CV for prediction. 
%
%
%
Yet, Agghoo has never been properly studied in the literature, even experimentally, 
to the best of our knowledge. 
The closest results we found \citep{Jun_Hu:2015,Jun:2016} 
study other procedures, called ACV and EKCV, 
and prove weaker theoretical results than ours; 
we explain in Section~\ref{sec.setting.agghoo} 
why Agghoo is more natural and should be preferred to ACV and EKCV in general. 

%
Because of the aggregation step, Agghoo is an ensemble method, 
and it is particularly close to subagging ---a variant of bagging where the bootstrap is replaced by subsampling--- 
\citep{Buh_Yu:2002}, 
since both combine subsampling with aggregation. 
However, Agghoo is \emph{not} subagging applied to the hold-out selected predictor, 
in which the subsample itself is split into training and validation samples; 
see Section~\ref{sec.setting.agghoo}. 
Therefore, we cannot apply previous results on subagging for analyzing Agghoo; new developments are required. 
Since bagging and subagging are well-known for their stabilizing effects \citep{Bre:1996a,Buh_Yu:2002},  
we can expect Agghoo to behave similarly; 
in particular, it should improve much the prediction performance of CV 
when the hold-out selected predictor is unstable. 
The simulation experiments of Section~\ref{Sec:SimusStud} confirm this intuition. 

\paragraph{Contributions} 
The purpose of this paper is to investigate both theoretical and practical performances of Agghoo. 
As a first step, we focus on supervised classification with the 0--1 loss. 
%
%
%
First, Section~\ref{sec:TheoGar} 
provides three theoretical results on the classification error of Agghoo. 
%
%
Proposition~\ref{Prop:AgBetterThanHO} is a general result 
---valid for any family of learning rules--- 
proving that the excess risk of Agghoo is smaller than the excess risk 
of hold-out selection, multiplied by the number of classes. 
This ensures that Agghoo can be used safely: 
one cannot loose too much by preferring Agghoo to the hold-out. 
%
%
We then focus on binary classification under the margin condition \citep{Mam_Tsy:1999}, 
which is well-known for allowing fast learning rates \citep{Aud_Tsy:2007,Lec:2007c,Aud:2009}. 
Theorem~\ref{hold_out} is a non-asymptotic oracle inequality for Agghoo, 
showing that Agghoo performs almost as well as the best possible selection rule (called the oracle). 
We illustrate the sharpness of this oracle inequality by considering the setting of 
\citet[Section~3]{Aud_Tsy:2007}:  
Theorem~\ref{minimax_ho} shows that Agghoo applied to a well-chosen family of classifiers 
yields a minimax-adaptive procedure ---hence, optimal in worst-case. 
%
%
%
%
Finally, Section~\ref{Sec:SimusStud} is a numerical study of the performance of Agghoo, 
on synthetic and real data sets, 
with two different kinds of selection problems. 
It illustrates that Agghoo actually performs much better than hold-out, 
and even better than CV ---provided its parameters are well-chosen. 
When choosing among decision trees, 
the prediction performance of Agghoo is much better than the one of CV 
---for the same computational cost---, 
which illustrates the strong interest of using Agghoo when the choice among 
competing learning rules is ``unstable''. 
Based upon our experiments, we also give in Section~\ref{Sec:SimusStud} 
some guidelines for choosing Agghoo's parameters: 
the training set size and the number of data splits.

\section{Setting and definitions} \label{sec.setting}


\subsection{Supervised classification} \label{sec.setting.classif}
Consider the supervised classification problem 
where a sample $(X_1,Y_1)$, $\ldots$, $(X_n,Y_n) \in \cX \times \cY$ is given, 
with $\cX$ any measurable space and $\mathcal{Y} = \{0, \ldots ,M-1\}$ for some integer $M\geq 2$. 
The goal is to build a classifier ---that is, a measurable map $f:\cX \to \cY$--- 
such that, for any new observation $(X,Y)$, $f(X)$ is a good prediction for $Y$. 
Throughout the paper, 
$(X_1, Y_1), \ldots, (X_n,Y_n), (X,Y)$ are assumed independent with the same distribution $P$. 
Let $\cF$ denote the set of classifiers. 
The quality of any classifier $f \in \cF$ is measured by its risk 
\[
R(f) = \bP_{(X,Y)\sim P} \bigl( Y\neq f(X) \bigr)
\]
and any optimal classifier 
$ \bayes\in \argmin_{f\in \cF}R(f)$ 
is called a Bayes classifier \citep{Dev_Gyo_Lug:1996}. 
We also define the excess risk of $f \in \cF$ by 
$\ell(\bayes,f)=R(f)-R(\bayes)=R(f)-\inf_{f\in \cF}R(f) \geq 0$. 

Since we only have access to $P$ through the data $D_n := (X_i,Y_i)_{1 \leq i \leq n}$, 
a classifier is built from data, 
thanks to a \emph{classification rule} 
$\ClassRule:\cup_{k\ge 1} (\cX \times \cY)^k \to \cF$, 
which maps a sample (of any size) into a classifier. 
%

\subsection{Selection of classification rules by cross-validation} \label{sec.setting.selec}

%
A generic situation is when a family $\cG$ of classification rules is given, 
and not only a single rule, 
so that we have to select one of them 
---or to combine their outputs. 
%
%
For instance, we can consider the family 
$(G_k^{\mathrm{NN}})_{k \geq 1}$ of nearest neighbors classifiers 
when $\cX$ is a metric space ---the parameter $k$ denoting the number of neighbors---, 
or the family $(G_{\lambda}^{\mathrm{SVM}})_{\lambda \in [0,+\infty)}$ 
of support vector machine classifiers 
for a given kernel on $\cX$ ---$\lambda$ denoting the regularization parameter. 

\medbreak

%
A common answer to this problem is to \emph{select} one of these rules from data, 
and cross-validation methods are a general tool for doing so. 
Let us briefly recall these methods here; 
we refer the reader to the survey by \citet{Arl_Cel:2010:surveyCV} for details and references, 
and to the paper by \citet{Arl_Ler:2012:penVF:JMLR} for the most recent results.

%
For any non-empty $B \subset \{1, \ldots, n\}$, 
define the corresponding sample and empirical risk by 
\[ 
D_n^B := (X_i,Y_i)_{i\in B}
\qquad \text{and} \qquad 
\forall f \in \cF, \qquad 
\cRh_B(f) := \frac{1}{|B|} \sum_{i \in B} \un_{\{f(X_i) \neq Y_i\}} 
\, , 
\]
respectively. 
Then, for any classification rule $\ClassRule$, 
$\ERM{\ClassRule,B} := \ClassRule(D_n^B)$ is the (random) classifier 
obtained by training $\ClassRule$ on the subsample of $D_n$ indexed by $B$. 

%
Given a non-empty subset $T$ of $\{1, \ldots, n\}$ 
and some classification rule $\ClassRule$, 
the hold-out estimator of the risk of $\ClassRule$ is defined by 
$\HO{T}{\ClassRule} := \cRh_{T^{c}} (\ERM{\ClassRule,T})$. 
A classifier $\ERM{G,T}$ is built from the \emph{training sample} $D_n^T$, 
then its quality is assessed on the \emph{validation sample} $D_n^{T^c}$. 
Note that $\HO{T}{\ClassRule}$ depends on the sample $D_n$ 
even if this does not appear in the notation. 
Then, hold-out can be used for \emph{selecting} one rule among $\cG$, 
by minimizing $\HO{T}{\ClassRule}$ over $\ClassRule \in \cG$:  
\begin{equation}\label{def:ERMho}
\ERMho{\cG,T} 
:= \widehat{G}^{\ho}_{\cG,T}(D_n^T),
\qquad \text{where}\qquad 
\widehat{G}^{\ho}_{\cG,T} \in \argmin_{G\in \cG}\HO{T}{G}
\enspace .
\end{equation}
We call $\ERMho{\cG,T}$ the \emph{hold-out classifier}. 

%
Hold-out depends on the arbitrary choice of a training set $T$, 
and is known to be quite unstable, 
despite its good theoretical properties \citep[Section~8.5.1]{Mas:2003:St-Flour}. 
%
Therefore, practicioners often prefer to use cross-validation instead, 
which considers several training sets. 
%
%
Given a family $\cT = \{T_1, \ldots, T_V \}$ of training sets, 
the cross-validation risk estimator of any classification rule $\ClassRule$ 
is defined by 
\[
\text{CV}_{\cT}(G) = \frac{1}{V}\sum_{j=1}^V \HO{T_j}{G}
\enspace ,
\]
leading to the \emph{cross-validation classifier} 
\begin{equation}\label{def:CVEst}
\ERM{\cG, \cT}^{\mathrm{cv}} 
:= \widehat{G}^{\mathrm{cv}}_{\cG,\cT}(D_n) 
\qquad \text{where} \qquad 
\widehat{G}^{\mathrm{cv}}_{\cG,\cT}\in\argmin_{G\in \cG}\text{CV}_{\cT}(G)
\enspace .
\end{equation}
%
%
Depending on how $\cT$ is chosen, this can lead to 
leave-one-out, leave-$p$-out, $V$-fold cross-validation 
or Monte-Carlo cross-validation, among others \citep{Arl_Cel:2010:surveyCV}. 

\subsection{Aggregated hold-out (Agghoo)} \label{sec.setting.agghoo}
%
%
In this paper, we study another way to improve on the stability of hold-out classifiers, 
by \emph{aggregating} the hold-out classifiers $\ERMho{\cG,T_j}$ obtained from several training sets $T_1, \ldots, T_V$. 
Formally, the \emph{aggregated hold-out classifer} 
(Agghoo) is obtained by making a \emph{majority vote} among them: 
\begin{equation}\label{def:estim}
\forall x \in \cX, \qquad 
\ERMag{\cG, \cT}(x) \in \argmax_{y\in\cY} \frac{1}{V}\sum_{j=1}^V \un_{\left\{\ERMho{\cG,T_j}(x)=y \right\}}
\enspace. 
\end{equation}

%
Compared to cross-validation classifiers defined by Eq.~\eqref{def:CVEst}, 
we reverse the order between aggregation (majority vote or averaging) and minimization of the risk estimator. 
%
%
%
To the best of our knowledge, Agghoo has not appeared before in the literature. 
%
%
The closest procedures we found 
are ``$K$-fold averaging cross-validation'' (ACV) proposed by \citet{Jun_Hu:2015} 
for linear regression, 
and ``efficient $K$-fold cross-validation'' (EKCV) proposed by \citet{Jun:2016} 
for high-dimensional regression. 
The main difference with Agghoo is that ACV and EKCV average the chosen \emph{parameters} 
---the models for ACV, the regularization parameters for EKCV---, 
whereas Agghoo averages the chosen \emph{classifiers}. 
%
This leads to completely different procedures for learning rules that are not linear functions 
of their parameters. 
Even in the case of linear regression with ACV, 
where the estimator is a linear function of the projection matrix onto the model 
---which is the ``parameter'' averaged by ACV---, 
Agghoo would lead to a different procedure since it averages the 
$\widehat{G}^{\ho}_{\cG,T_j}(D_n^{T_j})$ 
while ACV averages the 
$\widehat{G}^{\ho}_{\cG,T_j}(D_n)$. 
In the general case, and in particular for classification, 
averaging the classifiers $\widehat{G}^{\ho}_{\cG,T_j}(D_n^{T_j})$, 
which have been selected for their good performance on the validation set $T_j^c$, 
is more natural than averaging the 
$\widehat{G}^{\ho}_{\cG,T_j}(D_n)$ whose performance has not been assessed on independent data. 

%
Another related procedure ---not explicitly studied in the literature--- is subagging applied 
to hold-out selection 
$D_n \mapsto \widehat{G}^{\ho}_{\cG,T}(D_n^T)$,
as defined by Eq.~\eqref{def:ERMho}. 
Denoting by $T_1, \ldots, T_V$ the subsample sets chosen by the subagging step, 
and by $T'_j$ the training set used by the hold-out applied to the subsample 
$D_n^{T_j}$, 
the subagged hold-out predictor is obtained by making a majority vote among 
$\ERM{\widehat{G}^{\ho}_{\cG,T'_j} (D_n^{T_j})} (D_n^{T_j})  
=  \widehat{G}^{\ho}_{\cG,T'_j} (D_n^{T_j})$. 
%
Compared to Agghoo, the main difference is that this aggregates predictors 
trained on a \emph{subsample} of $D_n^{T_j}$ ---instead of the entire~$D_n^{T_j}$.

\section{Theoretical guarantees}\label{sec:TheoGar}



In this section, we make the following two assumptions on the training sets: 
\begin{gather}
\label{hyp.T-ind}
\tag{H1}
T_1, \ldots, T_V \subset \{1, \ldots, n\} 
\text{ are independent of } 
D_n = (X_i,Y_i)_{1 \leq i \leq n} 
\, , 
\\
\label{hyp.T-p}
\tag{H2}
|T_1| = \cdots = |T_V| = n-p 
\in \{1, \ldots, n-1\}
\, . 
\end{gather}
%
%
These assumptions are satisfied for the most classical cross-validation methods, 
including leave-$p$-out, $V$-fold cross-validation (with $p=n/V$) and 
Monte-Carlo cross-validation \citep{Arl_Cel:2010:surveyCV}. 
%
%

\paragraph{General bound} 
Our first theoretical result connects the performances of Agghoo and hold-out.
\begin{proposition}\label{Prop:AgBetterThanHO}
Let $\cG$ denote a collection of classification rules and $\mathcal{T}$ a collection of training sets 
satisfying \eqref{hyp.T-ind}--\eqref{hyp.T-p}. 
The aggregated hold-out estimator $\ERMag{\cG, \cT}$ defined by Eq.~\eqref{def:estim} 
and the hold-out estimator defined by Eq.~\eqref{def:ERMho} satisfy: 
\[ 
\bE \bigl[ \ell(\bayes,\ERMag{\cG, \cT}) \bigr] 
\leq M \bE \bigl[ \ell(\bayes,\ERMho{\cG, T_1}) \bigr] 
\qquad \text{and} \qquad 
\bE \bigl[ R(\ERMag{\cG, \cT}) \bigr] 
\leq 2 \bE\bigl[ R(\ERMho{\cG, T_1}) \bigr] 
\enspace .
\]
\end{proposition}
Proposition~\ref{Prop:AgBetterThanHO} is proved in supplementary material. 

Up to the multiplicative constant $M$, 
the Agghoo predictor $\ERMag{\cG, \cT}$ performs theoretically at least as well as the hold-out selected estimator. 
When the number of classes $M$ is small ---for instance, in the binary case---, 
this guarantees that Agghoo's prediction error 
cannot be too large. 
Let us now use this general bound to obtain oracle and minimax properties for Agghoo.


%
\paragraph{Oracle inequality in binary classification} 
From now on, we focus on the binary classification case, that is, $M=2$. 
A classical assumption in this setting is the so-called \emph{margin assumption} 
\citep{Mam_Tsy:1999}, for some $\beta \geq 0$ and $c\ge 1$: 
\begin{equation}\label{hyp.MA} \tag{MA}
\forall h>0, \qquad 
\mathbb{P} \bigl( \bigl\lvert 2\eta(X) - 1 \bigr\rvert \leq h \bigr) \leq c h^{\beta}
\qquad \text{where} \qquad 
\eta(X) := \mathbb{P} (Y = 1 | X ) 
\enspace ;  
\end{equation}
$\eta$ is called the regression function. 
Agghoo satisfies the following non-asymptotic oracle inequality under \eqref{hyp.MA}. 
\begin{theorem}
\label{hold_out}
Let $\cG$ denote a collection of classification rules and $\mathcal{T}$ a collection of training sets 
satisfying \eqref{hyp.T-ind}--\eqref{hyp.T-p}. 
If $\beta \geq 0$ and $c \ge 1$ exist such 
that \eqref{hyp.MA} holds true, 
we have: 
\[ 
\bE \bigl[ \ell(\bayes,\ERMag{\cG, \cT}) \bigr] 
\leq 3 \bE \left[ \inf_{G \in \cG}\ell(\bayes,\ERM{G,T_1}) \right] 
 + \frac{ 29 c^{\frac{1}{\beta + 2}}  \log(e|\cG|)}{p^{\frac{\beta + 1}{\beta + 2}}}  
 \enspace.
\]
\end{theorem}
Theorem~\ref{hold_out} is proved in supplementary material. 
Note that the constant $3$ in front of the oracle excess risk 
$\bE [ \inf_{G \in \mathcal{G}}\ell(\bayes,\ERM{G,T_1}) ] $ 
can be made as close as desired to~$2$, at the price of enlarging the 
constant $29$ in the remainder term. 

Theorem~\ref{hold_out} shows that Agghoo 
performs as well as the best classification rule in the collection $\cG$ trained with ${n-p}$ data 
(called the oracle classifier), 
provided that the remainder term 
$\cO(\log(e|\cG|)/p^{\frac{\beta + 1}{\beta + 2}})$ can be neglected. 
This is a strong result, since it shows that Agghoo attains the same learning rates 
as the oracle classifier. 
In comparison, the results proved by \citet{Jun_Hu:2015} on ACV are much weaker, 
since they do not allow to derive any rate for ACV, except in the ``parametric'' setting. 


\paragraph{Minimax rates} 
We now address the main possible limitation of Theorem~\ref{hold_out}: 
is the remainder term 
negligible in front of the oracle excess risk in some ``interesting'' frameworks? 
To this end, we consider the setting of \citet[Section~3]{Aud_Tsy:2007}, 
which is an example where fast minimax rates are known under the margin assumption. 
Let us briefly describe this setting. 
For any $\gamma, L>0$, let $\Sigma_{\gamma,L}$ denote the class of 
functions with $\gamma$-H\"older semi-norm smaller than $L$ 
(a formal definition is given in supplementary material). 
The classes of probability distributions that we'll be interested in is defined as follows.
\begin{definition}
For any $\gamma, L > 0$, $\beta \geq 0$ and $c \ge 1$, 
$\mathcal{P}_{\gamma,L,\beta,c}$ denotes the set of probability distributions
$P$ on $\mathbb{R}^d \times \{0 , 1\}$ such that 
\eqref{hyp.MA} holds true ---with parameters $\beta,c$---, 
the regression function $\eta \in \Sigma_{\gamma,L}(\bR^d)$ 
and $X$ has a density $q$ with respect to the Lebesgue measure satisfying: 
\[
\frac{1}{L} \un_{[0;1]^d} 
\leq q(X) 
\leq L \un_{[0;1]^d} 
\qquad \text{almost surely.} 
\]
\end{definition}
By \citet[Theorem 3.5]{Aud_Tsy:2007}, the minimax rate of convergence over 
$\mathcal{P}_{\gamma,L,\beta,c}$ is $n^{- \frac{\gamma (1 + \beta)}{2\gamma + d}}$ 
if $\gamma \beta \leq d$; 
this result is recalled in supplementary material. 
When $\gamma \beta < d$, we have $\gamma (1 + \beta)/(2\gamma + d)<(\beta + 1)/(\beta + 2)$, 
so the remainder term in Theorem~\ref{hold_out} can be neglected in front of 
the minimax rate of convergence over $\mathcal{P}_{\gamma,L,\beta,c}$, 
if $\cG$ is a \emph{polynomial collection} ---that is, $|\cG| \leq n^{\alpha}$ for some $\alpha \geq 0$---  
and if $p \geq \delta n$ with $\delta >0$. 

As detailed by \citet[Section~3]{Aud_Tsy:2007}, having $\gamma \beta > d$ is a strong constraint on 
$P$, hence we can leave it aside in the following without losing too much. 
%
We can also remark that when $L=+\infty$ ---which removes the condition on the density $q$ of $X$---,  \citet[Section~4]{Aud_Tsy:2007} show that the minimax rate then is 
$n^{- \frac{\gamma (1 + \beta)}{\gamma(2 + \beta) + d}}$ 
---even when $\gamma \beta \geq d$---, 
and this rate is always larger than the remainder term in Theorem~\ref{hold_out} 
if $\cG$ and $p$ are taken as previously. 

%
Let us now explain how Agghoo can be used for building a minimax classifier,  
simultaneously over all classes $\mathcal{P}_{\gamma,L,\beta,c}$ such that $\gamma \beta < d$. 
The construction relies on a family of classification rules proposed by 
\citet[Definition~2.3]{Aud_Tsy:2007}. 
For any $\ell \geq 1$ and $h > 0$, 
$G^{\mathrm{LP}}_{\ell,h}$ is a plug-in classification rule 
based upon a local-polynomial Gaussian kernel estimator 
$\widehat{Q}_{\ell,h}$ of $\eta$, 
where $\ell$ is the polynomial degree 
and $h$ is the kernel bandwidth. 
Then, $G^{\mathrm{LP}}_{\floor{\gamma},h_n(\gamma,d)}$ is minimax over $\mathcal{P}_{\gamma,L,\beta,c}$, 
where $h_n(\gamma,d)$ is a well-chosen bandwidth. 
%
%
This result, and the full definition of $G^{\mathrm{LP}}_{\ell,h}$, 
are recalled in supplementary material. 
%
Its main limitation is that the minimax rule depends on $\gamma$ 
which is unknown; 
in other words, it is not \emph{adaptive} to the smoothness $\gamma$ 
of the regression function. 
The next theorem shows that Agghoo is a simple way 
to get rid of this issue. 
\begin{theorem}
\label{minimax_ho}
Let $\tau \in (0,1)$ be fixed.  
For any $n \geq 1$, let 
$\mathcal{T}_n=\{T_1,\ldots,T_V\}$ be a collection of training sets 
satisfying \eqref{hyp.T-ind}--\eqref{hyp.T-p} and $|T_1|=\floor{\tau n}$, 
for some $V \geq 1$ (possibly depending on $n$), 
and let us define the collection 
\[ 
\cG^{\mathrm{LP}}_n 
:= \left( G^{\mathrm{LP}}_{\ell,\frac{1}{k}} \right)_{1 \leq \ell \leq n, 1 \leq k \leq n} 
\]
of classification rules, 
where $G^{\mathrm{LP}}_{l,h}$ are defined above. 
Then, the Agghoo classifier $\ERMag{\cG^{\mathrm{LP}}_n , \cT_n}$ defined by Eq.~\eqref{def:estim} 
is minimax over classes $\mathcal{P}_{\gamma,L,\beta,c}$ 
simultaneously for all $\gamma,L,\beta,c$ such that $\gamma \beta < d$. 
\end{theorem}
Theorem~\ref{minimax_ho} is proved in supplementary material. 
It shows that Agghoo applied to a well-chosen collection of 
local polynomial estimators defined by \citet{Aud_Tsy:2007} 
yields a \emph{minimax adaptive} classifier. 
Let us emphasize that minimax adaptivity is a strong theoretical property.  

Of course, Agghoo is not the only way to obtain such an adaptivity result 
---for instance, hold-out selection is sufficient to get it, 
as can be seen by taking $V=1$ in Theorem~\ref{minimax_ho}---, 
even if we have not found such a minimax-adaptive statement for this problem in the literature. 
The main goal of Theorem~\ref{minimax_ho} is to illustrate that our oracle inequality (Theorem~\ref{hold_out}, 
from which the minimax adaptivity of Agghoo derives) 
is sharp, and that Agghoo can lead to optimal classifiers, 
in one non-trivial setting at least.

\section{Numerical experiments}\label{Sec:SimusStud}


This section illustrates the performance of Agghoo and CV on synthetic and real data. 
%

\subsection{$k$-nearest neighbors classification} \label{sec.simus.kNN}
We first consider the collection $\cG^{\mathrm{NN}} = (G_k^{\mathrm{NN}})_{k \geq 1, \, k \text{ odd}}$ 
of nearest-neighbors classifiers ---assuming $k$ is odd for avoiding ties--- 
on the following binary classification problem. 

\paragraph{Experimental setup} 
%
%
Data $(X_1,Y_1), \ldots, (X_n,Y_n)$ are independent, 
with $X_i$ uniformly distributed over $\cX = [0,1]^2$ 
and 
\[
\bP (Y_i = 1 | X_i) 
= \sigma\left(\frac{g(X_i) - b}{\lambda}\right) 
\quad \text{where} \quad 
\sigma(u) = \frac{1}{1 + e^{-u}} \, , 
\ 
g(u,v) = e^{-(u^2 + v)^3} + u^2 + v^2
\, , 
\]
$b = 1.18$ and $\lambda = 0.05$; 
the parameters have been chosen to get a challenging but feasible problem. 
%
The Bayes classifier is $\bayes: x \mapsto \un_{g(x) \geq b}$ 
and the Bayes risk, computed numerically using the scipy.integrate python library, 
is approximately equal to $0.242$. 
%
%
%
Agghoo and CV are used with the collection $\cG^{\mathrm{NN}}$ 
with training sets $T_1, \ldots, T_V$ that are chosen independently and uniformly 
among the subsets of $\{1, \ldots, n\}$ with cardinality $\floor{\tau n}$,  
for different values of $\tau$ and~$V$; 
hence, CV corresponds to what is usually called ``Monte-Carlo CV'' \citep{Arl_Cel:2010:surveyCV}. 
%
%
Each algorithm is run on $1000$  
independent samples of size $n=500$, 
and independent test samples of size $1000$ are used 
for estimating the 0--1 excess risks 
$\ell(\bayes,\ERMag{\cG^{\mathrm{NN}},\cT})$, $\ell(\bayes, \ERM{\cG^{\mathrm{NN}}, \cT}^{\mathrm{cv}} ) $ 
and the oracle excess risk $\inf_{G \in \cG^{\mathrm{NN}}} \ell(\bayes, G(D_n))$. 
%
%
Expectations of these quantities are estimated by taking an average over the $1000$ samples; 
we also compute standard deviations for these estimates, 
which are not shown on Figure~\ref{AgghooVsCV} since they are all smaller than $3.6 \%$ of the 
estimated value, so that all visible differences on the graph are significant. 

%
\paragraph{Results} are shown on Figure~\ref{AgghooVsCV}. 
%
%
%
The performance of Agghoo strongly depends on both parameters $\tau$ and $V$. 
%
%
Increasing $V$ improves significantly the performance of the resulting estimator.  
However, most of the improvement seems to occur between $V = 2$ and 
$V = 10$, so that taking $V$ much larger seems useless, 
a behaviour previously observed for CV \citep{Arl_Ler:2012:penVF:JMLR}. 
%
%
%
As a function of $\tau$, the performance of Agghoo is U-shaped although not symmetric around $1/2$. 
It seems that $\tau\in [0.5,0.8]$ yields better performances, 
while taking $\tau$ close to $0$ or $1$ should be avoided (at least for $V \leq 20$). 
Taking $V$ large enough, say $V=10$, makes the choice of $\tau$ less crucial: 
a large region of values of $\tau$ yield (almost) optimal performance. 
%
%
Similar conclusions can be drawn for CV, with the major difference that its 
performance depends much less on $V$: only $V=2$ appears to be significantly worse than $V \geq 5$. 

%
Let us now compare Agghoo with the hold-out (that is, $V=1$) and CV. 
For a given $\tau$, Agghoo and CV are much better than the hold-out; 
there is no surprise here, considering several data splits is always useful. 
For fixed $(\tau,V)$, Agghoo does significantly better than CV if $V \geq 10$, worse if $V=2$, 
and they yield similar performance for $V=5$. 
Overall, if we can afford the computational cost of $V=10$ data splits, 
Agghoo with optimized parameters ($V=10$, $\tau \in [0.5 , 0.8]$) 
clearly improves over CV with optimized parameters ($V=10$, $\tau=0.7$). 
This advocates for the use of Agghoo instead of CV, unless 
we have to take $V < 5$ for computational reasons.

\subsection{CART decision trees} \label{simu:agghoo_cart}

In a second experiment, we consider the CART classification trees of \citet{Bre_etal:1984}. 
%
%

\paragraph{Experimental setup} 
%
%
Data samples of size $n=500$ are generated, following an experiment of \citet[Section~2]{Gen_Pog_Tul:2010}: 
%
%
%
%
$Y \sim \mathcal{B}(0.5)$ and $\varepsilon \sim \mathcal{B}(0.7)$ 
are independent Bernoulli variables, 
$\cX = \R^d$ with $d \geq 6$ 
and, given $(Y,\varepsilon)$, 
the coordinates $X^{(j)}$, $j=1, \ldots, d$ of $X$ 
are independent with the following distribution. 
When $\varepsilon = 1$, 
$X^{(j)}\sim\mathcal{N}(j Y,1)$ for $j \in \{1,2,3\}$ 
and $X^{(j)} \sim \mathcal{N}(0,1)$ otherwise. 
When $\varepsilon = 1$, 
$X^{(j)}\sim\mathcal{N}((j-3) Y,1)$ for $j \in \{4,5,6\}$ 
and $X^{(j)} \sim \mathcal{N}(0,1)$ otherwise. 
We estimate numerically that the Bayes risk is $0.041$.  

%
Agghoo and 10-fold CV are used with the family $(G^{\mathrm{CART}}_{\alpha})_{\alpha >0}$ 
of pruned CART classifiers, with training sets $T_j$ chosen as in Section~\ref{sec.simus.kNN}. 
%
%
Let us detail how $G^{\mathrm{CART}}_{\alpha}$ is defined. 
For any $\alpha>0$ and any sample $D_n$, the $\alpha$-pruned CART classifier 
$G^{\mathrm{CART}}_{\alpha}(D_n)$ is defined as the partitioning classifier 
built on the tree 
\[ 
\widehat{t}_{\alpha,n} 
= \argmin_{t \subset \widehat{t}} \bigl\{ \cRh(\ERM{t}) + \alpha |\cL_t|  \bigr\}
\enspace,
\]
where $\widehat{t}$ denotes the fully grown CART tree, 
$\ERM{t}$ is the partitioning classifier associated with any tree $t$, 
$\cRh = \cRh_{\{1, \ldots, n \}}$ is the empirical risk on the whole data set $D_n$ 
and $|\cL_t|$ is the number of leaves (terminal nodes) of the tree $t$. 
%
%
%
%
%
%
%
%
We also consider random forests (RF), which are a natural competitor to Agghoo applied on pruned CART trees, 
using the R package \texttt{randomForest} \citep{Lia_Wie:2002} with default values for all parameters. 
Since RF typically build hundred of trees ---500 by default---, whereas Agghoo only requires to build around 
$V=10$ trees, we also consider RF with $\mathtt{ntree}=10$ trees (and default values for all other parameters). 
%
%
%
%
%
The rest of the experimental protocole is the same as in Section~\ref{sec.simus.kNN}, 
in particular we consider $1000$ samples of size $n=500$, 
and risks are estimated with test samples of size $1000$. 
Standard deviations of our risk estimations are not shown since they never exceed $2.5 \%$ of the excess risk. 

%

\begin{figure}
\begin{minipage}[t]{.56\linewidth}
\centering
\vspace*{-0.5cm} 
\includegraphics[width=\textwidth]{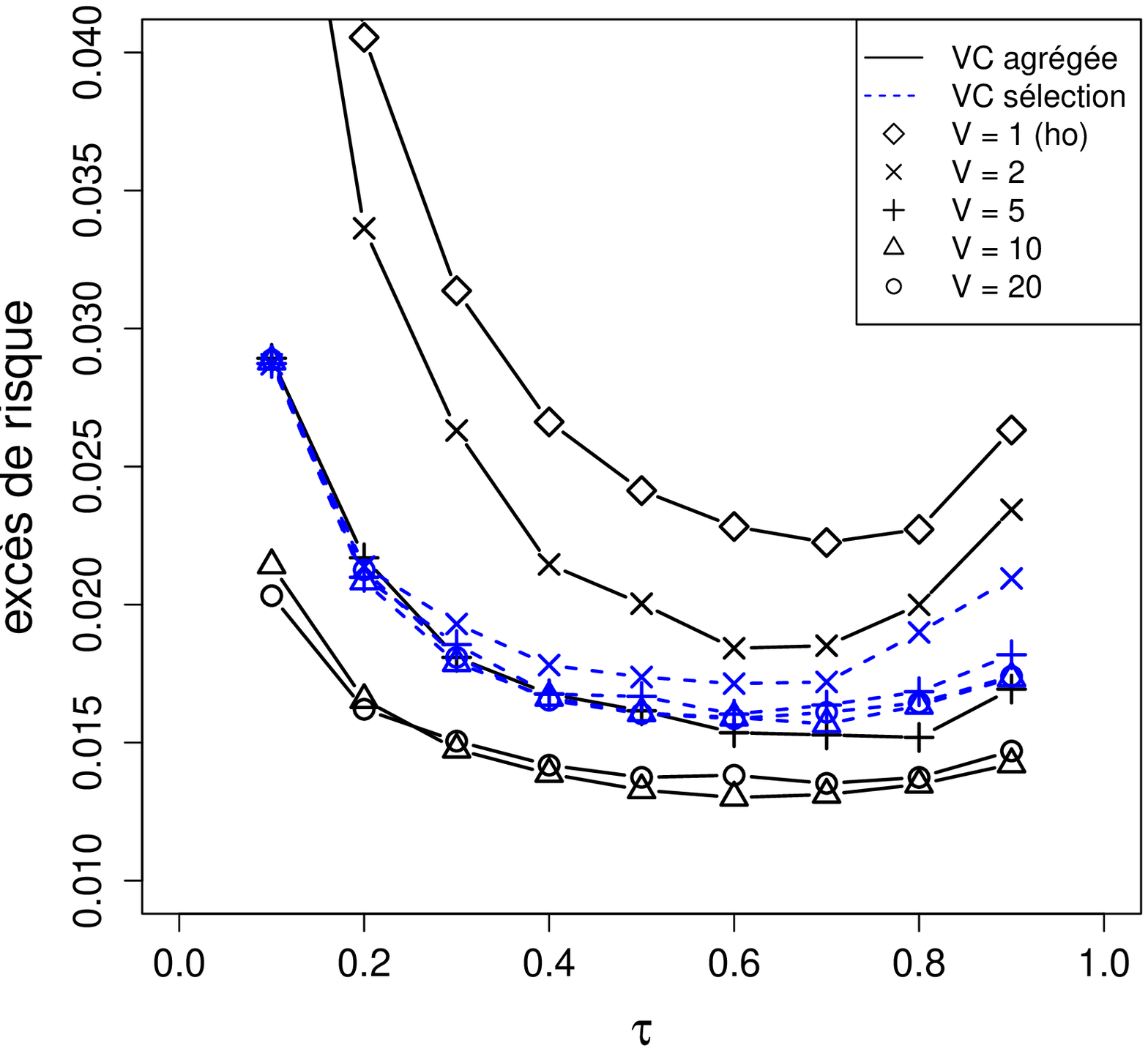} 
\vspace*{-0.5cm} 
  \caption{\label{AgghooVsCV} Classification performance of Agghoo and CV 
for the $k$-NN family; 
the oracle excess risk is $0.0034 \pm 0.0004$
     \vspace*{-0.5cm} 
}
\end{minipage}
\hspace*{.02\linewidth}
\begin{minipage}[t]{.4\linewidth}
\centering
\vspace*{-0.7cm} 
\includegraphics[scale=0.3]{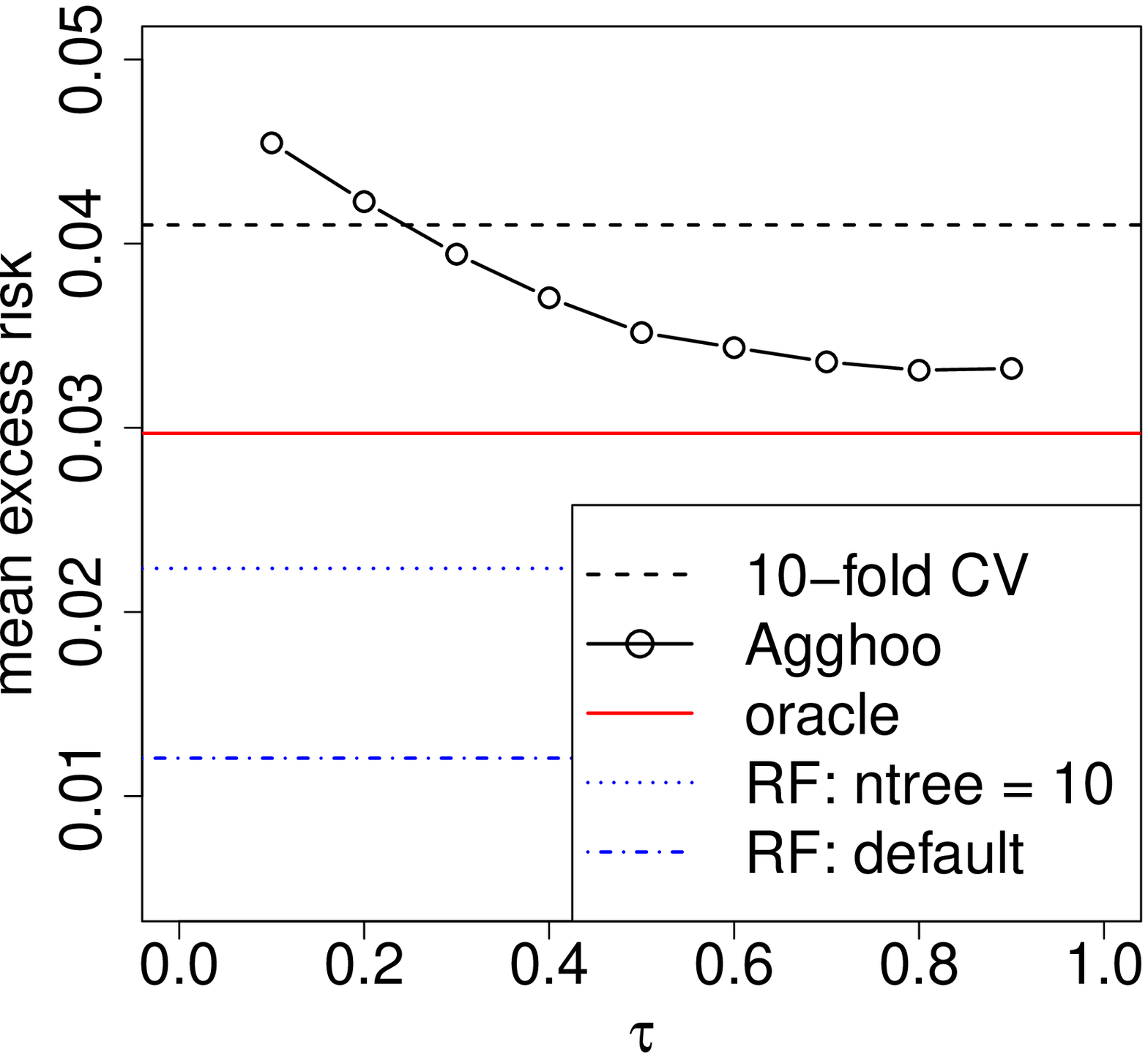}
\vspace*{-0.5cm} 
\\
\includegraphics[scale=0.3]{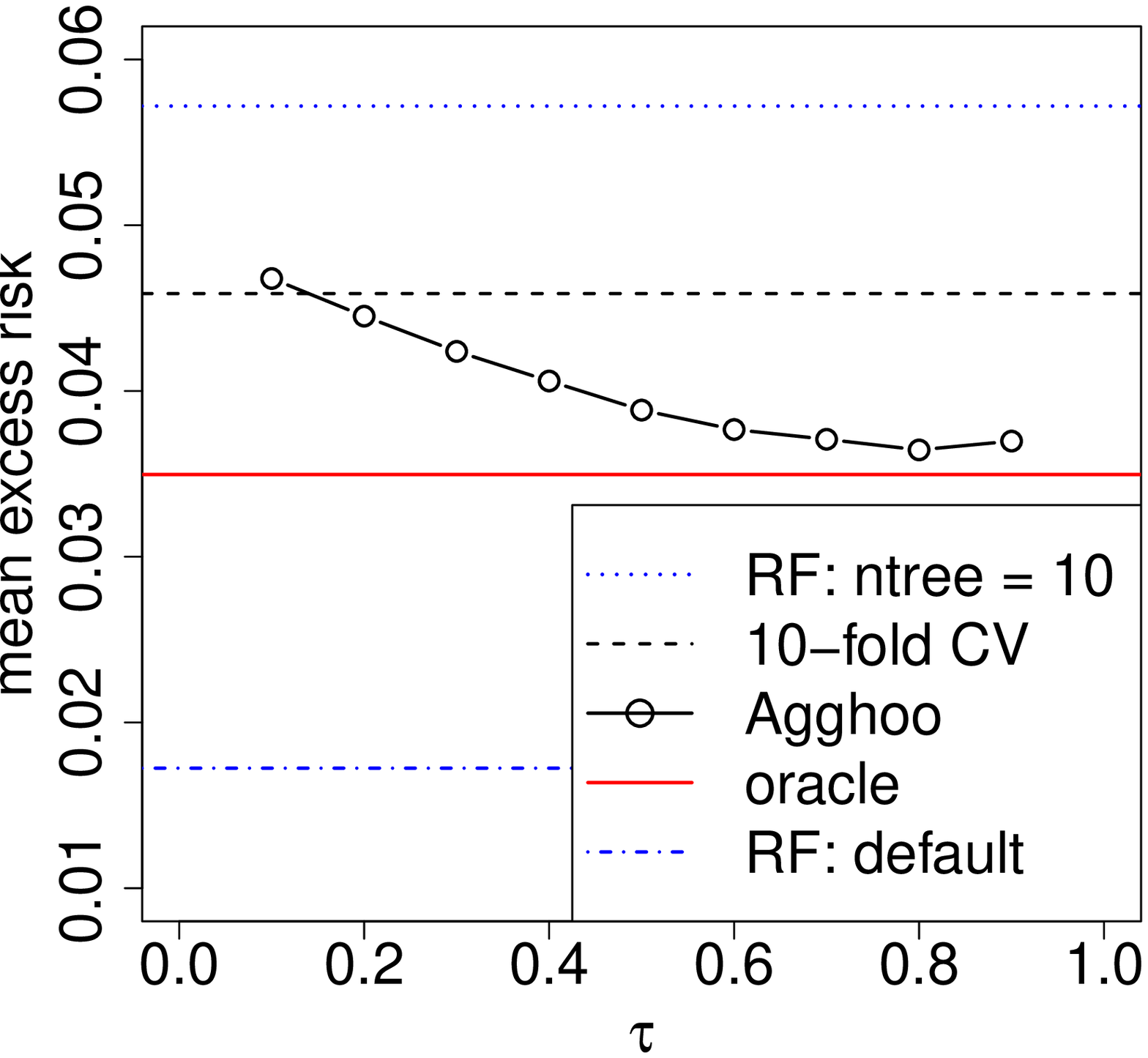}
\vspace*{-0.2cm} 
\caption{ \label{fig.cart-synth} Synthetic dataset (top: $d=7$, bottom: $d=50$); 
classification performance of Agghoo ($V=10$) and $10$-fold CV built on the CART family, 
as well as the oracle choice and RF 
}
\end{minipage}
\end{figure}
\paragraph{Results} 
are shown on Figure~\ref{fig.cart-synth}, with $d=7$ (only one noise variable) 
and $d=50$ (large majority of noise variables). 
%
%
Concerning Agghoo and CV, the conclusions are similar to the $k$-NN case, 
studied in Section~\ref{sec.simus.kNN}. 
This is why we only report here their performances for $V=10$. 
%
Let us only notice that Agghoo with $\tau = 0.8$ and $V=10$ performs almost as well as the oracle, 
which is a strong competitor since it uses the knowledge of data distribution~$P$. 
%
%
The main novelty here is the comparison to RF: 
RF with 500 trees clearly outperform all other procedures (and the oracle!), 
whereas RF with 10 trees works well only for $d=7$. 
We explain this fact by noticing that 
RF builds trees with more randomization than Agghoo 
(especially when many variables are pure noise, in the $d=50$ case), 
so that individual trees (and small forests) perform worse than hold-out selected classifiers, 
whereas a large enough forest performs much better. 
Therefore, when such a specific algorithm is available, our advice is to prefer it to Agghoo; 
but in the general case, it might be difficult to add some randomization individual classifiers, 
so that Agghoo remains useful. 
Agghoo can also be interesting when we cannot afford to aggregate more than a few tens of classifiers. 

\paragraph{Real-data experiment} 
We also consider the same procedures on the breast-cancer Wisconsin dataset from UCI 
Machine Learning Repository \citep{Lic:2013}, 
for which $n=699$ and $d=10$. 
%
We use samples of size $500$ and the remaining $199$ data as test samples. 
The split between sample and test sample is done $1000$ times, 
so that we can provide error bars for the estimated risks reported in Table~\ref{tab.CART-real}. 
%
%
\begin{table}[t]
  \caption{Results for the breast-cancer Wisconsin dataset, CART-based classifiers. }
  \label{tab.CART-real}
  \centering
  \begin{tabular}{ll}
    \toprule
    Procedure      & 0--1 risk in \%  
    \\
    \midrule
    $10$-fold CV                       & $6.66  \pm 0.06$   \\ 
    Agghoo ($V = 10$, $\tau = 0.8$)    & $5.36  \pm 0.05$  \\
    \bottomrule
  \end{tabular}
\qquad \qquad 
  \begin{tabular}{ll}
    \toprule
    Procedure      & 0--1 risk in \%  
    \\
    \midrule
    Oracle                             & $5.08 \pm 0.04$ \\
    RF  ($\mathtt{ntree} = 10$)        & $3.68 \pm 0.04$ \\
    RF  ($\mathtt{ntree} = 500$)       & $3.01 \pm 0.03$ \\
    \bottomrule
  \end{tabular}
%
%
\end{table}
The conclusion is similar to what has been obtained on synthetic data with $d=7$, 
which confirms our interpretation: 
when most variables are relevant, RF's individual trees are less randomized 
so that RF can perform well even with a small number of trees.

\section{Discussion} \label{sec.concl}

%
The theoretical and numerical results of the paper 
show that Agghoo can be used safely in supervised classification with the 0--1 loss, 
at least when its parameters are properly chosen ---$V \geq 10$ and $\tau \in [0.5 , 0.8]$ 
seem to be safe choices. 
On the theoretical side, Agghoo performs at least as well as the hold-out 
and it satisfies some sharp oracle inequality. 
Experiments show that Agghoo actually performs much better, 
improving over cross-validation except when the number of data splits 
is strictly smaller than~$5$, 
especially when basis classifiers are unstable ---like decision trees. 
Proving theoretically that Agghoo can improve over CV is an exciting, but challenging, open problem 
that we would like to address in future works. 

%
So, for the same computational cost, 
Agghoo ---with properly chosen parameters $V,\tau$--- should be preferred to CV, 
unless the final classifier has to be interpretable (which makes selection better than aggregation). 
Yet, Agghoo and CV only are general-purpose tools. 
For some specific families of classifiers, better procedures can be used. 
For instance, when one wants a classifier built upon decision trees, 
random forests ---with sufficiently many trees if there are many irrelevant variables--- 
certainly are a better choice, as shown by our experiments. 

%
Our results can be extended in several ways. 
%
%
First, our theoretical bounds directly apply to subagging hold-out, 
since it also makes a majority vote among hold-out selected estimators. 
The difference is that in subagging the training set size is $n-p-q$ and the validation set size is $q$, 
for some $q \in \{1, \ldots, n-p-1\}$, 
leading to slightly worse bounds than the ones we obtained for Agghoo (not much if $q$ is well chosen). 
Oracle inequalities and minimax results can also be obtained for Agghoo in other settings; 
this paper focuses on two such results for length reasons only. 
%
%
%
Second, Agghoo can be extended to other learning problems, such as regression and density estimation, 
replacing the majority vote by an average. 
Proposition~\ref{Prop:AgBetterThanHO} then holds true with $M$ replaced by $1$ if the loss is convex 
---for instance, for the least-squares loss---, by Jensen's inequality. 
Then, Theorem~\ref{hold_out} could be proved based upon 
the general results of \citet[Section~8.5]{Mas:2003:St-Flour}. 
%
%
%
Third, Agghoo might be improved by weighting the votes of the different hold-out classifiers, 
as suggested by \citet{Jun:2016}.

{\small 
\bibliographystyle{abbrvnat}
\bibliography{biblio_nips2017}

\begin{thebibliography}{29}
\providecommand{\natexlab}[1]{#1}
\providecommand{\url}[1]{\texttt{#1}}
\expandafter\ifx\csname urlstyle\endcsname\relax
  \providecommand{\doi}[1]{doi: #1}\else
  \providecommand{\doi}{doi: \begingroup \urlstyle{rm}\Url}\fi

\bibitem[Arlot and Celisse(2010)]{Arl_Cel:2010:surveyCV}
S.~Arlot and A.~Celisse.
\newblock A survey of cross-validation procedures for model selection.
\newblock \emph{Statistics Surveys}, 4:\penalty0 40--79, 2010.

\bibitem[Arlot and Lerasle(2016)]{Arl_Ler:2012:penVF:JMLR}
S.~Arlot and M.~Lerasle.
\newblock Choice of {$V$} for {$V$}-fold cross-validation in least-squares
  density estimation.
\newblock \emph{Journal of Machine Learning Research (JMLR)}, 17\penalty0
  (208):\penalty0 1--50, 2016.

\bibitem[Audibert(2009)]{Aud:2009}
J.-Y. Audibert.
\newblock Fast learning rates in statistical inference through aggregation.
\newblock \emph{The Annals of Statistics}, 37\penalty0 (4):\penalty0
  1591--1646, 2009.

\bibitem[Audibert and Tsybakov(2007)]{Aud_Tsy:2007}
J.-Y. Audibert and A.~Tsybakov.
\newblock {Fast learning rates for plug-in classifiers}.
\newblock \emph{Annals of Statistics}, 35\penalty0 (2):\penalty0 608--633,
  2007.

\bibitem[Biau and Devroye(2015)]{Bia_Dev:2015}
G.~Biau and L.~Devroye.
\newblock \emph{Lectures on the Nearest Neighbor Method}.
\newblock Springer Series in the Data Sciences. Springer, 2015.

\bibitem[Biau and Scornet(2016)]{Bia_Sco:2016:TEST}
G.~Biau and E.~Scornet.
\newblock A random forest guided tour.
\newblock \emph{TEST}, 25\penalty0 (2):\penalty0 197--227, 2016.

\bibitem[Breiman(1996)]{Bre:1996a}
L.~Breiman.
\newblock Bagging predictors.
\newblock \emph{Machine Learning}, 24\penalty0 (2):\penalty0 123--140, 1996.

\bibitem[Breiman(2001)]{Bre:2001}
L.~Breiman.
\newblock Random forests.
\newblock \emph{Machine Learning}, 45:\penalty0 5--32, 2001.

\bibitem[Breiman et~al.(1984)Breiman, Friedman, Olshen, and
  Stone]{Bre_etal:1984}
L.~Breiman, J.~H. Friedman, R.~A. Olshen, and C.~J. Stone.
\newblock \emph{{Classification and Regression Trees}}.
\newblock {Wadsworth Statistics/Probability Series}. Wadsworth Advanced Books
  and Software, Belmont, CA, 1984.

\bibitem[B{\"u}hlmann and van~de Geer(2011)]{Buh_vdG:2011}
P.~B{\"u}hlmann and S.~van~de Geer.
\newblock \emph{Statistics for high-dimensional data}.
\newblock Springer Series in Statistics. Springer, Heidelberg, 2011.
\newblock Methods, theory and applications.

\bibitem[B{\"u}hlmann and Yu(2002)]{Buh_Yu:2002}
P.~B{\"u}hlmann and B.~Yu.
\newblock {Analyzing bagging}.
\newblock \emph{The Annals of Statistics}, 30\penalty0 (4):\penalty0 927--961,
  2002.

\bibitem[Burnham and Anderson(2002)]{Bur_And:2002}
K.~P. Burnham and D.~R. Anderson.
\newblock \emph{{Model Selection and Multimodel Inference}}.
\newblock Springer-Verlag, New York, second edition, 2002.
\newblock A practical information-theoretic approach.

\bibitem[Devroye et~al.(1996)Devroye, Gy{\"o}rfi, and Lugosi]{Dev_Gyo_Lug:1996}
L.~P. Devroye, L.~Gy{\"o}rfi, and G.~Lugosi.
\newblock \emph{A Probabilistic Theory of Pattern Recognition}, volume~31 of
  \emph{{Applications of Mathematics (New York)}}.
\newblock Springer-Verlag, New York, 1996.

\bibitem[Dietterich(2000)]{Die:2000}
T.~G. Dietterich.
\newblock Ensemble methods in machine learning.
\newblock In \emph{International workshop on multiple classifier systems},
  pages 1--15. Springer, 2000.

\bibitem[Freund and Schapire(1997)]{Fre_Sch:1997}
Y.~Freund and R.~E. Schapire.
\newblock A decision-theoretic generalization of on-line learning and an
  application to boosting.
\newblock \emph{Journal of Computer and System Sciences}, 55\penalty0 (1, part
  2):\penalty0 119--139, 1997.
\newblock EuroCOLT '95.

\bibitem[Genuer et~al.(2010)Genuer, Poggi, and Tuleau-Malot]{Gen_Pog_Tul:2010}
R.~Genuer, J.-M. Poggi, and C.~Tuleau-Malot.
\newblock Variable selection using random forests.
\newblock \emph{Pattern Recognition Letters}, 31\penalty0 (14):\penalty0
  2225--2236, 2010.

\bibitem[Jung(2016)]{Jun:2016}
Y.~Jung.
\newblock Efficient tuning parameter selection by cross-validated score in high
  dimensional models.
\newblock \emph{International Journal of Mathematical, Computational, Physical,
  Electrical and Computer Engineering}, 10\penalty0 (1):\penalty0 19 -- 25,
  2016.

\bibitem[Jung and Hu(2015)]{Jun_Hu:2015}
Y.~Jung and J.~Hu.
\newblock A {$K$}-fold averaging cross-validation procedure.
\newblock \emph{Journal of Nonparametric Statistics}, 27\penalty0 (2):\penalty0
  167--179, 2015.

\bibitem[Lecu\'e(2007)]{Lec:2007c}
G.~Lecu\'e.
\newblock Optimal rates of aggregation in classification under low noise
  assumption.
\newblock \emph{Bernoulli}, 13\penalty0 (4):\penalty0 1000--1022, 2007.

\bibitem[Liaw and Wiener(2002)]{Lia_Wie:2002}
A.~Liaw and M.~Wiener.
\newblock Classification and regression by randomforest.
\newblock \emph{R News}, 2\penalty0 (3):\penalty0 18--22, 2002.
\newblock URL \url{http://CRAN.R-project.org/doc/Rnews/}.

\bibitem[Lichman(2013)]{Lic:2013}
M.~Lichman.
\newblock {UCI} machine learning repository, 2013.
\newblock URL \url{http://archive.ics.uci.edu/ml}.

\bibitem[Mammen and Tsybakov(1999)]{Mam_Tsy:1999}
E.~Mammen and A.~B. Tsybakov.
\newblock {Smooth discrimination analysis}.
\newblock \emph{The Annals of Statistics}, 27\penalty0 (6):\penalty0
  1808--1829, 1999.

\bibitem[Massart(2007)]{Mas:2003:St-Flour}
P.~Massart.
\newblock \emph{{Concentration Inequalities and Model Selection}}, volume 1896
  of \emph{{Lecture Notes in Mathematics}}.
\newblock Springer, Berlin, 2007.
\newblock Lectures from the 33rd Summer School on Probability Theory held in
  Saint-Flour, July 6--23, 2003, With a foreword by Jean Picard.

\bibitem[Nemirovski(2000)]{Nem:2000}
A.~Nemirovski.
\newblock \emph{{Topics in Non-parametric Statistics}}, volume 1738 of
  \emph{{Lecture Notes in Math.}}
\newblock Springer, Berlin, 2000.

\bibitem[Scholkopf and Smola(2001)]{Sch_Smo:2001}
B.~Scholkopf and A.~J. Smola.
\newblock \emph{{Learning with Kernels: Support Vector Machines,
  Regularization, Optimization, and Beyond}}.
\newblock MIT Press, Cambridge, MA, USA, 2001.

\bibitem[Steinwart and Christmann(2008)]{Ste_Chr:2008}
I.~Steinwart and A.~Christmann.
\newblock \emph{Support vector machines}.
\newblock Information Science and Statistics. Springer, New York, 2008.

\bibitem[Tibshirani(1996)]{Tib:1996}
R.~Tibshirani.
\newblock Regression shrinkage and selection via the lasso.
\newblock \emph{Journal of the Royal Statistical Society. Series B.
  Methodological}, 58\penalty0 (1):\penalty0 267--288, 1996.

\bibitem[Tsybakov(2004)]{Tsy:2004}
A.~B. Tsybakov.
\newblock {Optimal aggregation of classifiers in statistical learning}.
\newblock \emph{The Annals of Statistics}, 32\penalty0 (1):\penalty0 135--166,
  2004.

\bibitem[Yang(2001)]{Yan:2001}
Y.~Yang.
\newblock Adaptive regression by mixing.
\newblock \emph{Journal of the American Statistical Association}, 96\penalty0
  (454):\penalty0 574--588, 2001.

\end{thebibliography}
}

\newpage

\appendix

\section{Proof of Proposition \ref{Prop:AgBetterThanHO}}
The proof immediately follows from the following convexity-type property of the majority vote, 
applied to $(\ERMho{\cG, T_j})_{1 \leq j \leq V}$, 
since $\bE \bigl[ \ell(\bayes,\ERMho{\cG, T_j}) \bigr] $ and $\bE \bigl[ R(\ERMho{\cG, T_j}) \bigr]$ 
do not depend on $j$ under assumptions \eqref{hyp.T-ind}--\eqref{hyp.T-p} 
(they only depend on $T_j$ through its cardinality $n-p$). 
\begin{proposition}
\label{maj_vote}
Let $(\ERM{i})_{1 \leq i \leq V}$ denote a finite family of classifiers and let 
$\ERM{}^{\text{mv}}$ be some majority vote rule: 
$\forall x \in \cX$, $\ERM{}^{\text{mv}}(x) \in \argmax_{m \in \cY} \lvert \{i \in [V] : \ERM{i}(x) = m \} \rvert $. 
%
Then, 
\[ 
\ell(\bayes,\ERM{}^{\text{mv}}) 
\leq \frac{M}{V} \sum_{i = 1}^V \ell(\bayes,\ERM{i})
\qquad \text{and}\qquad 
R(\ERM{}^{\text{mv}}) \leq  \frac{2}{V} \sum_{i = 1}^V R(\ERM{i}) 
\enspace.\]
\end{proposition}

\begin{proof}
For any $y \in \cY$, define $\eta_y: x \rightarrow \mathbb{P}[Y = y | X = x]$. 
Then, for any $f \in \cF$, $R(f) = \bE [ 1 - \eta_{f(X)} (X) ]$ hence 
$\bayes(X) \in \argmax_{y \in \cY} \eta_y(X) $ 
and 
\[ 
\ell(\bayes,f) 
= \bE \Bigl[ \max_{y \in \cY} \eta_y(X) - \eta_{f(X)}(X) \Bigr] 
= \bE \bigl[ \eta_{\bayes(X)}(X) - \eta_{f(X)}(X) \bigr] 
\enspace. 
\]
We now fix some $x \in \cX$ and define $\mathcal{C}_x(y) = \{i \in [V]: \ERM{i}(x) = y \}$ 
and $C_x = \max_{y \in \cY} |\mathcal{C}_x(y)|$.
Since $C_x M\ge \sum_{y \in \cY} |\mathcal{C}_x(y)| = V $,
it holds $C_x \geq V/M$.
On the other hand, by definition of $\ERM{}^{\text{mv}}$,
\begin{align*}
 \frac{1}{V} \sum_{i = 1}^V \bigl[ \underbrace{ \eta_{\bayes(x)}(x) - \eta_{\ERM{i}(x)}(x) }_{\geq 0} \bigr]
 &\geq \frac{C_x}{V} \bigl( \eta_{\bayes(x)}(x) - \eta_{\ERM{}^{\text{mv}}(x)}(x) \bigr) 
 \geq \frac{1}{M} \bigl( \eta_{\bayes(x)}(x) - \eta_{\ERM{}^{\text{mv}}(x)}(x) \bigr) 
\enspace.
\end{align*}
Integrating over $x$ (with respect to the distribution of $X$) 
yields the first bound.

For the second bound, 
fix $x \in \cX$ and define $\mathcal{C}_x(y)$ and $C_x$ as above. 
Let $y \in \cY$ be such that $\ERM{}^{\text{mv}}(x) \neq y$. 
%
%
Since $y$ occurs less often than 
$\ERM{}^{\text{mv}}(x)$ 
among $\ERM{1}(x) , \ldots, \ERM{V}(x)$, 
we have $|\mathcal{C}_x(y)| \leq V/2$. 
Therefore,
\[
\frac{1}{V} \sum_{i = 1}^V \un_{\{ \ERM{i}(x) \neq y \} }
= \frac{V - |\mathcal{C}_x(y)|}{V} 
\geq \frac{1}{2} 
\enspace.\]
Thus
\[ 
\ERM{}^{\text{mv}}(x) \neq y 
\implies  
\frac{1}{V} \sum_{i = 1}^V \un_{\{ \ERM{i}(x) \neq y \} }
\geq \frac{1}{2}
\enspace.\]
Hence, for any $y \in \cY$, 
\[ 
\un_{\{\ERM{}^{\text{mv}}(x) \neq y \} }
\leq \frac{2}{V} 
\sum_{i = 1}^V \un_{\{ \ERM{i}(x) \neq y \} }
\enspace.\]
Taking expectations with respect to $(x,y)$ yields $R(\ERM{}^{\text{mv}} ) \leq 2 V^{-1} \sum_{i = 1}^V R(\ERM{i})$.
\end{proof}

\section{Proof of Theorem \ref{hold_out}}
The proof relies on a result by \citet[Eq.~(8.60), which is itself a consequence of Corollary~8.8]{Mas:2003:St-Flour}, which holds true as soon as 
\begin{equation} \label{eq.hyp-Eq-8.60-Massart}
\forall f \in \cF, \qquad 
\Var \bigl( \un_{\{ f(X) \neq Y \}} - \un_{ \{\bayes(X) \neq Y\} } \bigr) 
\leq  \Bigl[ w \bigl(\sqrt{\ell(\bayes,f)} \bigr) \Bigr]^2 
\end{equation}
for some nonnegative and nondecreasing continuous function $w$ on $\R^+$, 
such that $x \mapsto w(x)/x$ is nonincreasing on $(0,+\infty)$ and $w(1) \geq 1$. 

Let us first prove that assumption \eqref{eq.hyp-Eq-8.60-Massart} holds true. 
On one hand, for any $f \in \cF$, 
\begin{multline}
\label{eq.maj-Var-incr}
\Var \bigl(\un_{\{  f(X) \neq Y\} } - \un_{ \{\bayes(X) \neq Y\} } \bigr)
\leq
\bE[|  \un_{ \{f(X) \neq Y\} } - \un_{ \{\bayes(X) \neq Y \}} |^2] \\
= \bE[ \un_{ \{f(X) \neq \bayes(X) \}} ] 
= \bE \bigl[ \lvert f(X) - \bayes(X)|]
\enspace.
\end{multline}
On the other hand, since we consider binary classification with the 0--1 loss, 
for any $f \in \cF$ and $h>0$, 
\begin{align*} 
\ell(\bayes, f) 
&= \bE \bigl[ |2\eta(X) - 1| \cdot |f(X) - \bayes(X)| \bigr] 
\qquad  &\text{by \citet[Theorem~2.2]{Dev_Gyo_Lug:1996}} 
\\
&\geq h \bE \bigl[ |f(X) - \bayes(X)| \un_{\{|2\eta(X) - 1| \geq h\}} \bigr] 
\qquad  &
\\
&\geq h \bE \bigl[ |f(X) - \bayes(X)|  - \un_{\{|2\eta(X) - 1| < h\}} \bigr] 
\qquad  &\text{since } |f-f^*| \leq 1
\\
&\geq h \bE \bigl[ |f(X) - \bayes(X)| \bigr] - c h^{\beta+1}
\qquad  &\text{by \eqref{hyp.MA}.} 
\end{align*}
This lower bound is maximized by 
taking 
\[ 
h = h_* := \left( \frac{ \bE \bigl[ |f(X) - \bayes(X)| \bigr] } { c(\beta+1) } \right)^{\frac{1}{\beta}} 
\, , 
\]
which belongs to $[0,1]$ since 
$c \geq 1$  
and $\bE \bigl[ |f(X) - \bayes(X)| \bigr] \leq 1$. 
Thus, we obtain  
\begin{align*} 
\ell(\bayes, f) 
\geq 
h_* \frac{\beta}{\beta+1} \bE \bigl[ |f(X) - \bayes(X)| \bigr]
= \frac{\beta}{(\beta+1)^{(\beta+1)/\beta}  c^{1/\beta}} \bE \bigl[ |f(X) - \bayes(X)| \bigr]^{(\beta+1)/\beta} 
\end{align*}
hence Eq.~\eqref{eq.maj-Var-incr} leads to 
\begin{align*}
\Var \bigl( \un_{ \{f(X) \neq Y \}} - \un_{ \{\bayes(X) \neq Y\} } \bigr)
\leq
\bE \bigl[ |f(X) - \bayes(X)| \bigr] 
&\leq 
\frac{\beta+1}{\beta^{\beta/(\beta+1)}} c ^{\frac{1}{\beta + 1}} \ell(\bayes,f)^{\frac{\beta}{\beta + 1}} 
\leq 2c ^{\frac{1}{\beta + 1}}\ell(\bayes,f)^{\frac{\beta}{\beta + 1}} 
\enspace.
\end{align*}
Therefore, Eq.~\eqref{eq.hyp-Eq-8.60-Massart} holds true with 
$w(u) = \sqrt{c_1} u^\frac{\beta}{\beta + 1}$ and $c_1=2c ^{\frac{1}{\beta + 1}}$, 
which statisfies the required conditions.
So, by \citet[Eq.~(8.60)]{Mas:2003:St-Flour}, for any $\theta \in (0,1)$, 
\begin{equation} \label{eq:ho_oracle}
\bE \bigl[ \ell(\bayes,\ERMho{\cG, T}) \,\vert\, D_n^T \bigr] 
\leq \frac{1+\theta}{1-\theta} \inf_{G \in \cG} \ell(\bayes,\ERM{G,T}) + \frac{\delta_*^2}{1-\theta} 
\left[ 2\theta + \log(e |\cG|)\left(\frac{1}{3} + \theta^{-1} \right) \right] 
\end{equation} 
where $\delta_*$ is the positive solution of the fixed-point equation 
$w(\delta_*) = \sqrt{p} \delta_*^2$, that is 
$\delta_*^2 = ( c_1 / p )^\frac{\beta + 1}{\beta + 2} 
$. 
%
Taking expectations with respect to the training data $D_n^T$, 
we obtain 
\begin{align*}
\bE \bigl[ \ell(\bayes,\ERMho{\cG, T}) \bigr] 
&\leq 
\frac{1+\theta}{1 - \theta} \bE\left[ \inf_{G \in \mathcal{G}} \ell(\bayes,\ERM{G,T}) \right] +\frac{2 c^\frac{1}{\beta + 2}}{1-\theta} \frac{2\theta + \log(e |\cG|)\left(\frac{1}{3} + \theta^{-1} \right)}{p^\frac{\beta + 1}{\beta + 2} }
\enspace.
\end{align*}
Now, by Proposition~\ref{Prop:AgBetterThanHO}, 
\[ 
\bE \bigl[ \ell(\bayes,\ERMag{\cG, \cT}) \bigr] 
\leq 2 \bE \bigl[ \ell(\bayes,\ERMho{\cG, T_1}) \bigr] 
\leq 2\frac{1+\theta}{1 - \theta} \bE \left[ \inf_{G \in \mathcal{G}} \ell(\bayes,\ERM{G,T}) \right] 
  + \frac{4 c^\frac{1}{\beta + 2}}{1-\theta} \frac{2\theta + \log(e |\cG|)\left(\frac{1}{3} + \theta^{-1} \right)}{p^\frac{\beta + 1}{\beta + 2} } 
\enspace.
\]
Taking $\theta = 1/5$ leads to the result. 
\qed

\section{Proofs of minimax results}

\subsection{Formal framework and result}

To begin with, let us recall the definition of minimax optimality. 
 \begin{definition}
 Let $\mathcal{P}$ be a class of probability distributions over $\cX \times \cY$.
 %
  A classifier $\ERM{n}$ is said to be minimax over the class $\mathcal{P}$ when
 its maximal excess loss on $\mathcal{P}$ satisfies
 \[ \sup_{P \in \mathcal{P}} \bE_{P} \left[ \ell(\bayes_P, \ERM{n}) \right] \le \kappa(\cP)\inf_{\ClassRule } \sup_{P \in \mathcal{P}} 
 \bE_P\left[\ell(\bayes_P,\ClassRule(D_n)))\right] \enspace.\]
 Any sequence $u_n$ such that $u_n^{-1}\inf_{G} \sup_{P \in \mathcal{P}} 
 \bE_P\left[\ell(\bayes_P,\ClassRule(D_n))\right]$ is bounded and bounded away from $0$ is called a minimax rate of convergence. 
 \end{definition}

We now define formally $\Sigma_{\gamma,L}$. 

\begin{definition}\label{Def:GammaHolder}
 For any $k = (k_1,...,k_d) \in \mathbb{N}^d$, any $x=(x_1,\ldots,x_d) \in \mathbb{R}^d$ and any sufficiently smooth function $f:\bR^d\to\bR$, 
 let 
 \begin{align*}
  |k|= \sum_{i = 1}^d k_i,\quad  k! = \prod_{i=1}^d k_i! ,\quad D^k f(x) = \frac{\partial^{|k|}}{\partial x_1^{k_1} ... \partial x_d^{k_d}} f(x_1,...,x_d),\quad x^k = \prod_{i = 1}^d x_i^{k_i}\enspace.
 \end{align*}
 For any $\alpha \in \mathbb{N}$, any $\cC^\alpha$-function $f$ and any point $x\in\bR^d$, let
 \[ Q_{x}^\alpha f(y) = \sum_{k:|k| \leq \alpha} \frac{D^k f(x)}{k!}(y-x)^k \enspace.\]
 The class of $\gamma$-smooth functions with constant $L$ is defined as
 \[\Sigma_{\gamma,L}(\mathbb{R}^d) = \left\{ f \in \cC^{\floor{\gamma}}\left( \bR^d, \bR \right): \quad
  \forall x,y\in\bR^d,\quad  |f(y) - Q_x^{\floor{\gamma}}(y)| \leq L \Norm{y-x}^\gamma \right\} \enspace.\]
 \end{definition}
 Let us conclude this section with the definition of the minimax classification rules.
 
 Estimators achieving the minimax rates over the classes $\mathcal{P}_{\gamma,L,\beta,c}$ were obtained by \citet{Aud_Tsy:2007}. Let us recall this construction.
 
 \begin{definition}\label{def:MinMaxRule}
Let $\ell\in\bN$, $h>0$ and let $K$ denote the standard gaussian Kernel. For any polynomial $Q$, let
\[
\cC_{h,x}(Q) = \sum_{i = 1}^n (Y_i - Q(X_i - x))^2 K \left(\frac{X_i - x}{h} \right)
\enspace.  
\]
Denote by
\[ 
\widehat{Q}_{\ell,h,x} 
= \argmin_{Q: \deg Q \leq \ell} \mathcal{C}_{h,x}(Q) 
\]
  if the minimum exists and is unique and $\widehat{Q}_{\ell,h,x}=0$ otherwise.

Let also $\bar{B}$ be the matrix $(B_{s_1,s_2})_{|s_1|,|s_2| \leq \floor{\gamma}}$
where
\[
B_{s_1,s_2} 
= \frac{1}{n h^d} \sum_{i = 1}^n \left( \frac{X_i - x}{h} \right)^{s_1 + s_2} K \left( \frac{X_i - x}{h} \right)
\enspace. 
\]
Let $\lambda_{min} 
= \sup \{ \lambda \in \mathbb{R} \,\vert\, 
\forall u \in \mathbb{R}^d \Norm{\bar{B} u} \geq \lambda \Norm{u} \} $. 
The local polynomial rule with degree $\ell$ and bandwidth $h$ is defined by 
$G^{\mathrm{LP}}_{\ell,h,x} (D_n): x\mapsto \un_{\left\{ \widehat{\eta}_{\ell,h}(x) \geq \frac{1}{2} \right\}} $, 
where
\begin{align*} 
\widehat{\eta}_{\ell,h}(x) 
&= 
\begin{cases}
\widehat{Q}_{\ell,h,x}(0) &\text{ if } \lambda_{min} \geq (\log n)^{-1}  \\
 0 &\text{ otherwise} 
\end{cases}
\enspace.
  \end{align*}
 \end{definition}

\begin{theorem}
\label{minimax} 
The rate $u_n = n^{- \frac{\gamma (1 + \beta)}{2\gamma + d}}$ 
is minimax over the classes $\mathcal{P}_{\gamma,L,\beta,c}$, 
for any $\beta \ge 0$, $\gamma > 0$, $L > 0$, $c \ge 1$ and $\gamma \beta < d$.
  
Furthermore, if $n^{\frac{1}{2\gamma + d}}h_n\in[\tau,\tau'] $ for some $\tau'>\tau>0$,
the estimator $G^{\mathrm{LP}}_{\floor{\gamma},h_n}(D_n)$
is minimax over $\mathcal{P}_{\gamma,L,\beta,c}$ for any $L > 0,c > 0, \gamma \beta < d$.
\end{theorem}
\begin{proof}
The lower bound is proved in \cite[Theorem 3.5]{Aud_Tsy:2007}.
\\
For the upper bound, we shall denote by $C_1,C_2,C_3$ functions of the parameters $\gamma,L,a,b,d$ but neither $h$ nor $n$ which may vary from line to line.
Eq~(3.7) in \cite[Theorem 3.2]{Aud_Tsy:2007} implies that
\[ 
\sup_{P \in \mathcal{P}_{\gamma,L,\beta,c}} P(|\widehat{\eta}_{\floor{\gamma},h}(x) - \eta(x)| \geq \delta)  
\leq C_1 \exp (-C_2 nh^d \delta^2)
\]
for $h, \delta$ 
such that $0 < h \leq L^2$ and $C_3 h^{\gamma} \leq \delta $.

Then our choice of $h_n$ yields,
for any $\delta \ge C_3 h_n^{\gamma}$,
\begin{equation}\label{eq:Goal}
 \sup_{P \in \mathcal{P}_{\gamma,L,\beta,c}} P(|\widehat{\eta}_{\floor{\gamma},h_n}(x) - \eta(x)| \geq \delta)  
 \leq C_1 \exp (-C_2 n^{\frac{2\gamma}{2\gamma + d}} \delta^2) 
 \enspace.
\end{equation}
Moreover, for $\delta < C_3 h_n^{\gamma}$, 
\[ 
n h_n^d \delta^2 
\leq C_3 n h_n^{2\gamma + d} 
\leq C_3 b^{2\gamma + d}
\enspace,
\]
therefore, inequality \eqref{eq:Goal} holds for all $\delta>0$ and \cite[Lemma~3.1]{Aud_Tsy:2007} concludes the proof.
%
%
\end{proof}

\subsection{Proof of Theorem \ref{minimax_ho}}
Let $k_n = \floor{n^{\frac{1}{2\gamma + d}}}$ and $h_n = 1/k_n$.
Then for $n > d$, the risk of the oracle is bounded from above by
\[ 
\bE \left[ \inf_{G \in \cG^{\mathrm{LP}}_n} \ell(\bayes,\ERM{G,T_1}) \right] 
\leq 
\inf_{G \in \cG^{\mathrm{LP}}_n} \bE[\ell(\bayes, \ERM{G,T_1})] 
\leq \bE \bigl[ \ell(\bayes, \ERM{G^{\mathrm{LP}}_{\floor{\gamma},h_n},T_1}) \bigr] 
\enspace.
\]
Moreover, 
\[ 
1=\floor{n^{\frac{1}{2\gamma + d}}}h_n\le n^{\frac{1}{2\gamma + d}}h_n\le (\floor{n^{\frac{1}{2\gamma + d}}}+1)h_n\le 1+h_n\le 2\enspace. 
\]
Therefore $1 \leq n^{\frac{1}{2\gamma + d}}h_n \leq 2$
and, by Theorem~\ref{minimax} and the fact that $n-p \geq \tau n-1\ge \tau n/2$ for $n \geq 2/\tau$, 
there exist constants $C_{\cdot}$ depending on various parameters indicated in subscript, but not on $n$ such that
\[ 
\bE[\ell(\bayes, \ERM{G^{\mathrm{LP}}_{\floor{\gamma},h_n},T_1})] 
\leq C_{\gamma,L,\beta,c}(n-p_n)^{- \frac{\gamma(1 + \beta)}{2\gamma + d}}
\leq C_{\gamma,L,\beta,c,\tau}n^{- \frac{\gamma(1 + \beta)}{2\gamma + d}} 
\enspace.  
\]
By Theorem~\ref{hold_out}, there exists a constant $C_{c,\beta}$ such that
\begin{align*}
 \bE \bigl[ \ell(\bayes, \ERMag{\cG^{\mathrm{LP}}_n , \cT_n}) \bigr] 
 &\leq C_{c,\beta} \left(\bE \bigl[ \ell(\bayes, \ERM{G^{\mathrm{LP}}_{\floor{\gamma},h_n},T_1}) \bigr] 
 	+  \frac{\log(e|\cG^{\mathrm{LP}}_n|)}{p_n^{\frac{\beta + 1}{\beta + 2}}} \right)
 \enspace.
\end{align*}
Since $p_n \geq (1-\tau)n$ and $|\cG^{\mathrm{LP}}_n| = n^{2}$, it follows that
\begin{equation}\label{eq:UpperBoundERMinimax}
\bE \bigl[ \ell(\bayes, \ERMag{\cG^{\mathrm{LP}}_n , \cT_n}) \bigr] 
\leq C_{\gamma,L,\beta,c,\tau} n^{- \frac{\gamma(1 + \beta)}{2\gamma + d}} 
 	+ C_{c,\beta}\frac{\log n}{n^{\frac{\beta + 1}{\beta + 2}}}
\enspace .
\end{equation}
Finally, since $\gamma \beta < d$, $(\beta + 1)/(\beta + 2) > [\gamma(1+\beta)]/[2\gamma + d]$, and therefore 
\[ 
\frac{\log n}{n^{\frac{\beta + 1}{\beta + 2}}} 
\leq C_{\beta,\gamma,d}n^{- \frac{\gamma(1 + \beta)}{2\gamma + d}}
\enspace . 
\]
It follows from Eq.~\eqref{eq:UpperBoundERMinimax} that the Agghoo estimator is minimax. 
\qed 


\end{document}